\documentclass[sn-mathphys-ay]{sn-jnl} 
\usepackage{graphicx}
\usepackage{amsthm,jr}
\usepackage{comment}

\usepackage{graphicx}%
\usepackage{multirow}%
\usepackage{amsmath,amssymb,amsfonts}%
\usepackage{amsthm}%
\usepackage{mathrsfs}%
\usepackage[title]{appendix}%
\usepackage{xcolor}%
\usepackage{textcomp}%
\usepackage{manyfoot}%
\usepackage{booktabs}%
\usepackage{algorithm}%
\usepackage{algorithmicx}%
\usepackage{algpseudocode}%
\usepackage{listings}%

\newtheorem{thm}{Theorem}

\newcounter{example}
\newenvironment{example}{\addtocounter{example}{1}\noindent\textbf{Example \theexample.} }{\par\bigskip}
\newcounter{remark}
\newenvironment{remark}{\addtocounter{remark}{1}\noindent\textbf{Remark  \theremark.} }{\par\bigskip}

\def\nnorm{\mbox{}\mbox{$\not\hspace{-0.3em}\scn$}}

\title[Inconsistent GMLE]{A mixture of a normal distribution with random mean and variance\par Examples of inconsistency of maximum likelihood  estimates}

\author*[1]{\fnm{Ya\hspace{-0.1em}'\hspace{-0.1em}acov} \sur{Ritov}}\email{yritov@umich.edu}

\affil*[1]{\orgdiv{Department of Statistics}, \orgname{University of Michigan}, \orgaddress{
\city{Ann Arbor}, 
\state{Michigan}, \country{USA}}}

\abstract{
    We consider the estimation of the mixing distribution of a normal distribution where both the shift and scale are unobserved random variables. We argue that in general, the model is not identifiable. We give an elegant non-constructive proof that the model is identifiable if the shift parameter is bounded by a known value. However, we argue that the generalized maximum likelihood estimator is inconsistent even if the shift parameter is bounded and the shift and scale parameters are independent. The mixing distribution, however,  is identifiable if we have more than one observations per any realization of the latent shift and scale.
}


\begin{document}
\maketitle

\section{Introduction}
\cite{BahadurMLE1958} considered nonparametric examples in which the MLE (maximum likelihood estimator) exists and but is not consistent. His the first example is simple. Consider the class of all bounded (e.g., by 2) densities on the interval: $$\scf=\{F: \text{ with a continuous density } f(x)\le 2\ind(|x|<1)\},$$ and \iid observations $Y_1,\dots,Y_n$ from $F$. Here, $\ind(\cdot)$ is the indicator function. The MLE is simple but not unique: any probability  density function $\hat f$ such that $\hat f(Y_i)\equiv 2$. This example is not too interesting, and there is a good version of the MLE. Let $\bbf_n$ be the empirical distribution function and $\hat F$ be any distribution with  continuous density $\hat f$ satisfying $\|\hat F-\bbf_n\|_\en<n^{-1/2}$ and $\hat f(Y_i)=2$. The distribuiton $\hat F$ is a weakly consistent MLE. 

His second estimator is more interesting, but much more complicated. Bahadur presents an a countable number of distributions, $\scp=\{P_\en,P_1,P_2,\dots\}$, supported on the positive integers suck that $\lim_{k\to\en}P_k=P_\en$,  the MLE, $\hat P=\hat P_{\hat k_n}$, exists, is unique, and $\hat k_n\cas \en $.

In general, we consider in this paper non-dominated families and thus we need to use the generalized maximum likelihood estimator (GMLE) as defined in \cite{kieferWolfowitz1956}: if $Y_1,\dots,Y_n$ are \iid sample from $P\in\scp$, then $\hat P\in\scp$ is a GMLE if 
 \eqsplit{
    \forall \ti P\in\scp:\quad \summ i1n \log \frac{d\hat P}{d(\hat P+\ti P)}(Y_i)  \ge \summ i1n \log \frac{d\ti P}{d(\hat P+\ti P)}(Y_i).
  }

In this short note, whose subtitle is \cite{BahadurMLE1958}'s title,  we present another example of inconsistent MLE. The example is the natural generalization of a problem consider recently in the context of empirical Bayes, cf. \cite{SahaGunt20,JiangZhang09,GreenshteinRitov22,ritov2024}, in which we observe a mixture of distribution convoluted with a normal distribution. If the variance of the normal distribution is known, the MLE exists, and is relatively easy to compute (see the discussion and references in \cite{ritov2024}). In our problem, we assume that both the location and the scale of the normal distribution are unknown random variables. We show that, depending on the assumptions on the support of the location and shift parameters the problem may be identifiable or unidentifiable. We show that if the location parameter is unlimited, then the model is unidentifiable even if we try to restrict the definition to a minimum class. However, we show using a surprising technique that the model is identifiable if the shift parameter is bounded. This raised the question whether the GMLE is consistent in this identifiable situation. We argue that the answer is negative. It is well defined but converges to a wrong distribution. Assuming that the location and shift are independent random variables does not qualitatively these statements. The situation is interesting since the GMLE is inconsistence if we use a valid assumption while it is consistent if this true assumption is ignored. A similar situation happens with the GMLE for star shape distributions.

\section{On the identifiability of the normal mixture model}

The normal error measurement model is probably the most basic model of statistical science. We observe an \iid sample $Y_1,\dots,Y_n$ from a random variable  $Y$, where $Y$ is a random measurement of a value $X$. Unlike most standard models, we consider $Y_i=X_i+S_i\eps_i$, where $Y_i$ is observed, but $(X_i,S_i,\eps_i)$, $i=1,2,\dots,n$ are unobserved \iid sample, $\eps_i\dist \scn(0,1)$ is independent of $(X_i,S_i)\dist\Pi$, and $\Pi$ is unknown. 
Thus,
 \eqsplit{
    Y\dist F(y;\Pi) &\triangleq \Pi\circ\Phi
    \\
    &\triangleq \int\Phi\bigl(\frac{y-x}{s}\bigr) d\Pi(x,s),
}    
where $\Phi$ is the standard normal distribution.

Let 
\eqsplit{
    &\scf(\sci,\scs) =\{\Pi\circ\Phi, \Pi\in\scp(\sci,\scs) \},
}
where $\scp(\sci,\scs)=\{\Pi:\;\supp\Pi\subseteq\sci\times\scs\}$, $\supp\Pi$ is the support of $\Pi$, $\scs\subseteq\R$, and $\scs\subseteq[0,\en)$. Let $\scf^*(\sci,\scs)$ and $\scp^*(\sci,\scs)$ be defined similarly but with the restriction $\Pi(X\mid S)$ is symmetric around 0.

We can consider three different statistical problems with this model in the background:
\begin{enumerate}
  \item Estimate $F$ using the fact that $F\in\scf(\sci,\scs)$ to improve over the standard and simple empirical distribution of the sample.
  \item Estimating $\Pi$ (i.e., deconvolution of $F$).
  \item Find the empirical Bayes estimator:
  \eqsplit{
        D(y)&= E(X\mid Y=y)=\frac{\int x\varphi\Bigl(\frac{y-x}{s}\Bigr)d\Pi(x,s)} {\int \varphi\Bigl(\frac{y-x}{s}\Bigr)d\Pi(x,s)},
        }
  where $\varphi$ is the standard normal density.
\end{enumerate}

These three statistical problems presents three different challenges for the GMLE. The GMLE of $F$, the GMLE of $\Pi$, and the direct GMLE of $D(\cdot)$, respectively. The GMLE as a procedure, although the same, may succeed in one task and fail in another.  

Clearly, every cumulative distribution function (cdf) $F$ of a real random variable $Y$  is the weak limit of the convolution of $F$ with $N(0,\sig^2)$,  $\sig^2 \dec 0$. In other words, $F\in \scf(\sci,\scs)$, for any $\scs\ni0$ and $\sci$  is large enough (possibly the real lines $\R$) such that $\supp F\subset \sci$. Moreover, it is easy to see that for any $a$ such that $\scs-a\subseteq\R_+$ (where $\R_-$ and $\R_+$ are the negative and positive reals, respectively),  $\scf(\R,\scs)\subseteq\scf(\R, \scs-a)$,   since for any $\sig\ge a$, $G*N(0,\sig^2)=(G*N(0,a^2))*N(0,\sig^2-a^2)$, where $*$ is the convolution operator.   However, this argument does not hold for $\scf(\sci,\scs)$ for  $\sci\ne\R$ since the support of $G*N(0,\sig^2)$ is $\R$.  Thus, we have an identifiability question to be resolved.

To obtain identifiability, we may consider a restricted parameter set. We say that a distribution $P$ does not have a normal component, $P\in\nnorm$, if it cannot be represented as a convolution of another distribution with a (proper) normal distribution. We then define
\eqsplit{
    \scp_{NN}(\sci,\scs) &= \{\Pi\in\scp(\sci,\scs),\; \Pi_X \in\nnorm\}
    \\
    \scp_{NCN}(\sci,\scs) &= \{\Pi\in\scp(\sci,\scs),\;  \forall s\in\scs:\;\Pi(X\mid S=s) \in\nnorm \}.
}
where $\Pi_X$ is the marginal distribution of the location.
Similarly, we define $\scf_{NN}(\sci,\scs)$ and $\scf_{NCN}(\sci,\scs)$

A simple test whether a distribution belongs to $\nnorm$ is to consider its Fourier transform. Thus, a distribution function $G$ with characteristic function $\ti g(\omega)$ is in $\nnorm$ if for no $\al>0$, $\ti g(\omega)\exp(\al\omega^2)$ is a characteristic function. One conclusion from this test is that if $G_1\in\nnorm$ and $G_2\not\in\nnorm$ then $\gamma G_1+(1-\gamma)G_2\in\nnorm$.  Using this test we are going to argue that the main tool to ensure identifiability is restricting the support of $\Pi_X$.

\begin{thm}\label{th:ident}
    Suppose $F(Y;\Pi)\in \scf(\sci,[a,b])$, $a<b$.  Then,
    \begin{enumerate}
      \item  For any $\bar\Pi\in\scp_{NCN}(\R,[a,b])$ with $\supp \bar\Pi=\R\times[\bar a, \bar b]$, $a\le \bar a<\bar b\le b$, there is a $ \Pi\in\scp_{NCN}(\R,[a,b])$, $\bar\Pi\ne\Pi$, such that $F(Y;\Pi)=F(Y;\bar\Pi)$.
      \item However, if $F(Y;\Pi)\in \scf^*([-c,c],\scs)$, $c<\en$, then $\Pi$ is identifiable.
    \end{enumerate}
\end{thm}
\begin{proof}

Suppose $\supp \bar\Pi=\R\times [\bar a,\bar b]$ with $a\le \bar a< \bar b\le b$. We construct $\Pi$ by ``wraping'' $\bar\Pi$ around the middle of its $S$ support:
\eqsplit{
    \Pi(x,s)&=\Bigl(\bar\Pi(x,s) + \bar \Pi\bigl(x,s+\frac12(\bar a+\bar b)\bigr)*N(0,\del(s)\bigr)\Bigr)
    \ind\bigl(s\in(\bar a,\frac12(\bar a+\bar b)\bigr),
    \\
    \text{where }\del(s)&=s(\bar a+\bar b)+\frac14(\bar a+\bar b)^2.
     }
By the test suggested above, $\Pi\in\scp_{NCN}(\R,[a,b])$ since its first component is with no Gaussian component, but $F(Y;\Pi)=F(Y;\bar\Pi)$.

We prove now the second part of the theorem. Suppose $\Pi^1\circ\Phi=\Pi^2\circ\Phi$, $\Pi^1\ne \Pi^2$. We take the $\Pi^i$s to be the cdf's of $(X,S)$. Let $\ti \pi(\omega\mid s)=\E\Bigl(e^{j\omega X}\mid S=s\Bigr)$, where $j=\sqrt{-1}$.   Consider now the characteristic function of $f$:
 \eqsplit{
    \ti f^i (\omega) &= \int \ti \pi^i(\omega\mid s)e^{-\frac12s^2\omega^2}d\Pi^i_S(s),\quad i=1,2,
  }
 where $\Pi^i_S$ is the marginal distribution of $S$.

 However, $X$ is bounded. For simplicity (and WLOG) we assume it has a symmetric distribution. Hence, it has a Fourier representation:
 \eqsplit{
&\text{The conditional density: }  &  \pi^i(x\mid s)&=\summ k0\en (-1)^k a_k^i(s)\cos(\pi k x/c)\ind(|x|\le c)
    \\
&\text{Its Fourier transform: }   & \ti \pi^i(\omega\mid s) &= \summ k{-\en}\en a_k^i(s)\frac{\sin(\omega c)}{\omega-k\pi/c},\quad i=1,2.
  }
Thus, the characteristic function of $Y$:
 \eqsplit{
    \summ k{-\en}\en \frac{\sin(\omega c)}{\omega-k\pi/c} \int a_k^1(s) e^{-\frac12s^2\omega^2}d\Pi^1_S(s)
    &= \summ k{-\en}\en \frac{\sin(\omega c)}{\omega-k\pi/c} \int a_k^2(s) e^{-\frac12s^2\omega^2}d\Pi^2_S(s),
  }
 or
\eqsplit{
    \summ k{-\en}\en \frac{1}{\omega-k\pi/c} \int a_k^1(s) e^{-\frac12s^2\omega^2}d\Pi^1_S(s)
    &= \summ k{-\en}\en \frac{1}{\omega-k\pi/c} \int a_k^2(s) e^{-\frac12s^2\omega^2}d\Pi^2_S(s),
  }
Since we are dealing with analytic functions. Let
 \eqsplit{
    A_k(\omega)=\int a_k^2(s) e^{-\frac12s^2\omega^2}d\Pi^2_S(s)-\int a_k^1(s) e^{-\frac12s^2\omega^2}d\Pi^1_S(s).
  }
We obtained
 \eqsplit{
    \summ k{-\en}\en \frac{A_k(\omega)}{\omega-k\pi/c}=0.
  }
This implies that $A_k(\omega)=0$, $\omega\in\bbc$, $k\in\bbz$. Otherwise we would have $$A_m(w)=-(\omega-m\pi/c)\sum_{k\ne m} \frac{A_k(\omega)}{\omega-k\pi/c},$$
where the both the RHS and the LHS are analytic functions, but only the RHS has poles.
Spelling out,
\eqsplit[bas]{
     \int a_k^1(s) e^{-\frac12s^2\omega^2}d\Pi^1_S(s)
    &=  \int a_k^2(s) e^{-\frac12s^2\omega^2}d\Pi^2_S(s),\quad, \omega\in\bbr, k\in\bbz.
  }
But, $a_0^i(s)\equiv (2b)^{-1}$ (a distribution function has mass 1), hence \eqref{bas} with $k=0$:
\eqsplit{
     \int  e^{-\frac12s^2\omega^2}d\Pi^1_S(s)
    &=  \int  e^{-\frac12s^2\omega^2}d\Pi^2_S(s) ,
}
Since the Laplace transform is 1-1:
\eqsplit[k0]{
     \Pi^1_S&= \Pi^2_S.
 }
But \eqref{k0} together with \eqref{bas} for $k\ne 0$ implies that
\eqsplit{
     a_k^1(\cdot) &= a_k^2(\cdot),\quad k\in\bbz.
  }
  which concludes the proof.
  \end{proof}

\section{The generalized maximum likelihood Estimator}

A simple observation that was made is that every distribution belongs to the weak closure of $\scf(\R, [0,b])$, and hence if the support of $\Pi$ is unbounded, then the GMLE among all distributions is the same as the GMLE  among $\scf(\R,[0,b])$. That is, the empirical distribution is the GMLE, which is a good estimator of $F$, but an inconsistent estimator of $\Pi$.

This is not a surprise, because we have already proven that $\Pi$ is unidentifiable within this model. But Theorem \ref{th:ident} shows that the model $\scf([-c,c],[0,b])$, $c<\en$, is identifiable. This leaves the possibility that the GMLE of $\Pi$ is consistent under this assumption.   However, the proof of identifiability was based on a non-constructive argument, and it is not clear how to apply it for        an estimator. The analysis shows that in some sense the GMLE is actually worse  under $\scf(\sci,\scs)$ with bounded $\sci$ than it is  under $\scf(\R,\scs)$. Assuming  $\scf(\R,\sci)$ yields  inconsistent GMLE of $\Pi$ , but the GMLE of $F$ is consistent, while assuming a bounded $\sci$  leaves the GMLE  of $\Pi$  inconsistent but, also, the GMLE of observed $F$ itself is inconsistent. In other words, suppose $F\in\scf([-c,c],[0,b])$, $c<\en$. The GMLE of $F$ is better assuming nothing (i.e., the empirical distribution function) than holding to the restricted family. Here is formal asymptotic and non-asymptotic statements. The interesting phenomena is comparing Part \ref{part1} of Theorem \ref{th:ident} to its Part \ref{part2}. Suppose it is known that $F\in\scf(\R_-,\scs)$. Assuming that we would obtain inconsistent GMLE, however since $\scf(\R_-,\scs)\subset \scf(\R,\scs)$ we could assume less, only that $F\in\scf(\R,\scs)$ and then  the GMLE is consistent. 
\begin{thm}\label{Th:consis}
Suppose it is considered that $F(Y;\Pi)\in\scf(\sci,\scs)$ with $\scs\ni 0$. Then:
\begin{enumerate} 
\item If $\sci=\R$ then the GMLE of $F$ is $\bbf_n$, the empirical distribution function of $Y_1,\dots,Y_n$, and it is consistent.\label{part1}
\item If $\sci=\R_-$ and $\scs=\{0,1\}$, then . \label{part2}
 \eqsplit{
    \hat \Pi(A\times B) &= \bbf_n(A\cap\R_-)\ind(0\in B)+\bbf_n(Y>0)\ind\bigl((0,1)\in A\times B\bigr)
    \\
    &\cip F(A\cap\R_-)\ind(0\in B)+P(Y>0)\ind\bigl((0,1)\in A\times B\bigr)
    \\
    \hat F(y) &= \begin{cases}
                     \bbf_n(y)+\bbf_n(Y>0)\Phi(y), & y<0
                     \\
                     \bbf_n(Y>0)\Phi(y),&y\ge 0.
                   \end{cases}
    \\
    &\cip  \begin{cases}
                     F(y)+P(Y>0)\Phi(y), & y<0
                     \\
                     P(Y>0)\Phi(y),&y\ge 0.
                   \end{cases}
  }
Thus, neither the GMLE or $\Pi$ nor that of $F$ are consistent.  
\item Suppose $F\in\scf^*([-c,c],[0,b])$ for  $c>b$, $F(c)<1$, and  let $\eta$ be the solution of $\eta=ce^{-c(c-\eta)/b^2}$.  Then:
 \eqsplit{
    \hat \Pi(X\in A) &= 
                          \frac12\bbf_n (Y\in A)+\frac12\bbf_n (Y\in -A),\quad \forall A\subset (-c+\eta,c-\eta) 
    \\
    \hat \Pi(X\in (c-\eta,c)) &= \frac12\bbf_n (Y\in (c-\eta,c))+\frac12\bbf_n (Y\in -(-c,-c+\eta))
    \\
    &\hspace{3em}+\frac12\bbf_n(|Y|>c).
  }

\end{enumerate}
\end{thm}
\begin{proof}
    Assume $\sci=\R$. Any observation $Y_i$ can be explained completely by $\Pi$ having a point mass at $(Y_i,0)$. The GMLE among $\scf(\R,\scs)$ is the GMLE among all distributions on the line, namely $\bbf_n$. This proves the first part.
     
    For the second part note that any negative observation $Y_i$ can be explained completely by $\Pi$ having a point mass at $(Y_i,0)$. On the other hand, the positive observations contribute to the likelihood only through a density  of a normal distribution centered at a non-positive point. Thus, we consider the GMLE with respect to the Lebesgue measure plus the counting measure on the observations;
\eqsplit{
    \ell(\Pi;Y_1\dots,Y_n) &= \sum_{i:\,Y_i\le 0} \log p_i +  \sum_{i: Y_i>0}\log\sum_j q_j\varphi(Y_i-x_j).
    \\
    &p_j,q_j\ge 0,\;\;\sum p_j+\sum q_j=1,\;\;x_j\le 0,\;.
  }
   If $Y_i>0$, $\varphi(Y_i-x_j)$ is maximized by having $x_j=0$ (which is the closest possible value of $X$). Then the maximization is done by have $p_j=1/n$ and $q_1=\bbf_n(Y>0)$. This proves the second part. 

  We turn to the third part. For simplicity we consider  $X$ which has a strictly positive density on $(-c,c)$.
The GMLE is permitted to have points of mass inside the $[-c,c]$ interval, and we write the likelihood to be with respect to the Lebesgue measure plus the counting measure on the observations within the $(-c,c)$ interval and their mirror image (since the distribution is symmetric):
  \eqsplit{
    \ell(\Pi;Y_1\dots,Y_n) &= \sum_{i:\,|Y_i|<c} \log p_i +  \sum_{i: |Y_i|>c}\log\sum_j \frac{p_j}{s_j}\Bigl(\varphi\bigl(\frac{Y_i-x_j}{s_j}\bigr)+\varphi\bigl(\frac{Y_i+x_j}{s_j}\bigr)\Bigr)
    \\
    &p_j\ge 0,\;\;\sum p_j=\frac12,\;\; 0\le x_j\le c,\;\; 0<s_j\le b.
  }
(The sum of the probabilities is only $1/2$ since any point mass of the mixing distribution has its mirror image.) Any $N(x_j,s_j^2)$, $x_j\ne Y_i$, $s^2_j> 0$, does not contribute to the likelihood with respect to this measure at a point $Y_i\in(-c,c)$. Any point mass at $x_j$, does not contribute to the likelihood at any observation with $Y_i\ne x_j$.

The true distribution, $F(Y;\Pi)$, is symmetric and supported on all of  $\R$. Clearly, $\hat\Pi$ has a point mass of $1/2n$ at $(\pm Y_i,0)$ for any $|Y_i|<c$. This part of $\hat\Pi$ explains well the observations within the interval $(-c,c)$, but only them. With probability greater than 0, there are, however, observations outside this interval, which are more or less with comparable number at both sides of the interval. The GMLE component due to the points outside  the interval is discrete. Each $x_j$ belonging to the     support of this component is maximizing
  \eqsplit{
    \ell_O&= \sum_{i: |Y_i|>c}\log\sum_j \frac{p_j}{s_j}\biggl(\varphi\Bigl(\frac{Y_i-x_j}{s_j}\Bigr)+\varphi\Bigl(\frac{Y_i+x_j}{s_j}\Bigr)\biggr).
  }
Now, 
 \eqsplit{
    \frac{\partial \ell_O}{\partial x_j} &=   \sum_{i: |Y_i|>c} \frac{\frac{p_j}{s_j^3}\Bigl((Y_i-x_j)\varphi\bigl(\frac{Y_i-x_j}{s_j}\bigr)-(Y_i+x_j)\varphi\bigl(\frac{Y_i+x_j}{s_j}\bigr)\Bigr)} {\sum_k\frac{p_k}{s_k}\Bigl(\varphi\bigl(\frac{Y_i-x_k}{s_k}\bigr)+\varphi\bigl(\frac{Y_i+x_k}{s_k}\bigr)\Bigr)}
    \\
    &= \sum_{i: |Y_i|>c} \bigl(\al_{ji}(Y_i-x_j)-\ti\al_{ji}(Y_i+x_j)\bigr),\quad\text{say}
    \\
    &= \sum_{i: |Y_i|>c} (\al_{ji}-\ti\al_{ji})Y_i - x_j\sum_{i: |Y_i|>c} (\al_{ji}+\ti\al_{ji}).
  }
Thus, a zero of the derivative can be at a point that satisfies  
 \eqsplit{
  x_j &= \frac{\sum_{i: |Y_i|>c} (\al_{ji}-\ti\al_{ji})Y_i}{\sum_{i: |Y_i|>c} (\al_{ji}+\ti\al_{ji})}. 
  }
(This result is the basis of the EM algorithm for such a mixture.)

 Considering the second derivative at a maximizer $x_j$
 \eqsplit{
    0\ge \frac{\partial^2 \ell_O}{\partial x_j^2} &=   \sum_{i: |Y_i|>c}\Bigl(\al_{ji}\frac{(Y_i-x_j)^2}{s_j^2}+\ti\al_{ji}\frac{(Y_i+x_j)^2}{s_j^2}\Bigr) -\sum_{i: |Y_i|>c} (\al_{ji}+\ti\al_{ji}).
  }
which implies that $x_j>c-b$. Otherwise all of the fractures would be greater than 1 and the expression on the right hands side would be positive. More generally, let $x_j>c-\eta$. Now, $\al_{ji}/\ti\al_{ji}=e^{Y_ix_j/s_j^2}$, Thus,
 \eqsplit{
  x_j &= \frac{\sum_{i: |Y_i|>c} |\al_{ji}-\ti\al_{ji}||Y_i|}{\sum_{i: |Y_i|>c} (\al_{ji}+\ti\al_{ji})}
  \\
  &\ge  \frac{\sum_{i: |Y_i|>c} |\al_{ji}+\ti\al_{ji}|(1-e^{-c(c-\eta)/b^2})|Y_i|}{\sum_{i: |Y_i|>c} |\al_{ji}+\ti\al_{ji}|}
  \\
  &\ge c- ce^{-c(c-\eta)/b^2}.
  }
Thus, we have established that $x_j\ge c-\eta$, where $\eta$ is the solution of $\eta=ce^{-c(c-\eta)/b^2}$.

\end{proof}

\section{Examples and remarks}
\begin{example}
  Suppose,  it is known that  $Y_1,\dots,Y_n$ are \iid from $F(Y;\Pi)\in\scf^*(0,1,z_{0.025})$, where $z_\al$ is the $1-\al$ quantile of the normal distribution. Let the true $\Pi$ be    a point mass at 0. We can compute $\eta=0.046$ and the GMLE of $\Pi$ converge to a normal distribution with its 0.05 tail wrapped into the intervals of 0.046 from the end points, very far from the true point mass.  
\end{example}

\begin{example}
  Let $n\to\en$ and $c_n\to \en$. Suppose it is known that $Y_{n1},\dots,Y_{nn}$ are \iid from $F_n\in\scf(0,b,c_n)$. Suppose under the truth, $P(S=1)=1$, while $X_{n1}$ has a uniform distribution on $(-c_n,-c_n+1)\cup(c_n-1,c_n)$. Let $\ti Y_{n1}=|Y_{n1}|-c_n$. Then $\ti Y_{n1}\cid G(y)\equiv\int_{-1}^0 \Phi(y-z)dz$. Similarly, let $\ti X_{n1}=|X_n|-c_n$.

By considering $\ti W_{n1}$ we can concentrate on the behavior at the two end points, which, since $c_n\to\en$ become independent. In that case  $\eta=\eta_n\to 0$ and the second component of the GMLE becomes asymptotically two  point mass.  The GMLE of $\Pi_{\ti X}$ is, asymptotically, $G(x)\ind(x<0)+\bigl(1-G(0)\bigr)\del_{0}$. 

It follows from what was discussed so far that $S$ conditional on $\ti X<0$ is asymptotically concentrated on 0. However, when $\ti X=0$ the  GMLE of the distribution of  $S$ is asymptotically found to maximize the Kullback-Leibler affinity between $G(y)\ind(y>0)$ and a mixture of normal distributions with 0 mean. I.e., 
 \eqsplit{
    \argmax_H\int_0^\en \log\Bigl(\int_0^b \frac{1}{s}e^{-y^2/2s^2}d H(s) \Bigr)dG(y).
  }
  
\begin{remark}
  The situation of $\sci=\R_-$ is just purifying the $\sci=[-c,c]$ when $c\to\en$. It disassociates what happens on outside the interval from the right from the observations on the left of the interval. 
\end{remark}

\begin{remark}
The analysis was based on the assumption of one observation per latent variable. In many cases we have more than one. For simplicity, let assume that we have unobserved $(X_1,S_1),\dots,(X_n,S_n)$ \iid from $\Pi$, and for each $i$ we observe $(Y_{i1},Y_{i2})$ \iid $\scn(X_i,S^2)$. In that case the density is 
 \eqsplit{
    f(y_1,y_2\mid \Pi) &= \int \frac{1}{2\pi \sig^2}e^{-\frac{1}{2\sig^2}(y_1-x)^2-\frac{1}{2\sig^2}(y_2-x)^2}d\Pi(x,\sig)  
    \\
    &\le \int \frac{1}{2\pi \sig^2}e^{-\frac{1}{4\sig^2}(y_1-y_2)^2}d\Pi(x,\sig)
    \\
    &\le \frac{2}{\pi(y_1-y_2)^2}.  
  }
Thus the observation are from a mixture of a bounded two parameters full rank exponential family and the GMLE (actually, MLE) exists, is unique and consistent.
\end{remark}
\end{example}

\begin{remark}
  The analysis so far was based on the assumption that $(X,S)$ may have any distribution on $\sci\times\scs$. Suppose we assume that $S$ and $X$ are independent, $\Pi(A,B)=\Pi_X(A)\Pi_S(B)$. This is equivalent to the assumption that the distribution of $Y$ is the convolution of $\Pi_X$ with $G(y)=\int_\scs s^{-1}\Phi(y/s)d\Pi_S(s)$. The distribution $G$ is unimodal and symmetric. However it is not strongly unimodal. 
  
  Qualitatively, this would make a little difference if $\scs=\R_-$ and $0\in\scs$ (or, more generally, in its closure), since $G$ may have a point mass at 0. Thus, the GMLE is still defined on the counting measure on the negative observations plus Lebesgue. Let $G$ have a point mass $q$ at 0 and Lebesgue density $(1-q)g$. The likelihood is  
   \eqsplit{
    \ell(q,g, (p_1,x_1),\dots, ; Y_1,\dots,Y_n) &= \sum_{Y_i<0}\log (p_iq) + \sum_{Y_i>0} \log\bigl(\sum_j(1-q)p_j g(Y_i-x_j)\bigr).
    \\
    x_j\le 0, \sum p_j=1.  
    }
As it was in the general case the mass at 0 of the $S$ distribution contributes nothing to the positive observations (since $Y_i>0\ge x_j$) and the continuous part does not contribute anything to the observations where the dominating measure has a point mass. Thus, $\hat q=N_i/n$ trivially, where $N_i$ is the number of negative observations. Since, $g$ is unimodal and $Y_i-x_j>0$, the second term is maximized by having $j=1$ and $x_1=0$, thus $g$ maximize: $\sum_{Y_i>0}\log g(Y_i)$. As a simple example, consider $\scs=\{0,1\}$, then the maximum likelihood estimator converges weakly,
 $$\hat F\cip F(0)F(y\mn 0)+\bigl(1-F(0)\bigr)\int_{-\en}^0 \Phi(y-x)dF(x)+\bigl(1-F(0)\bigr)\Phi(y),$$
 where $a\mn b$ is the minimum of $a$ and $b$. 
 
 For another example, suppose $\sci=\R_-$ and $\sci=[0,2]$, while the true distribution is $N(0,1)$. In that case, the limiting  $g$ would be the standard normal, since it would give a perfect fit to the $Y$ distribution on $\R_+$, and thus the limit would be  different from the true $F_0(y)=\Phi(y)$:
  \eqsplit{
    \hat F(y)\cip& \frac12 \Phi(y\mn 0)+\frac12\int_{-\en}^0 \Phi(y-x)\varphi(x)dx+\frac12\Phi(y)
    \\
    =&  \frac12 \Phi(y\mn 0) + \frac14 - \frac14\Phi^{2}(-y/\sqrt2)+\frac12\Phi(y)
   }
Where we used the fact that the derivative of the integral in the first expression satisfies:
  \eqsplit{
    \int_{-\en}^0 \varphi(y-x)\varphi(x)dx &= \frac{1}{2\pi}\int_{-\en}^0 e^{-(y-x)^2/2-x^2/2}dx
    \\
     &= \frac{1}{2\pi}\int_{-\en}^{-y/2} e^{-(y/2-t)^2/2-(y/2+t)^2/2}dt
     \\
     &= \frac{1}{2\pi}e^{-y^2/4}\int_{-\en}^{-y/2} e^{-t^2}dt
     \\
     &= \frac{1}{2\sqrt{2}\pi}e^{-y^2/4}\int_{-\en}^{-y/\sqrt{2}} e^{-u^2/2}du 
     \\
     &=  \frac{1}{4\pi}\frac{\partial}{\partial y}\Bigl(\int_{-\en}^{-y/\sqrt{2}} e^{-u^2/2}du\Bigr)^2 
     \\
     &= -\frac{1}{2} \frac{\partial}{\partial y}\Phi^2(-y/\sqrt{2}).
   }
   
\end{remark}

\begin{remark}
  Suppose $\scs=\R$ but $\sci=[a,b]$. The model is unidentifiable as stated in Theorem \ref{th:ident}. This is relevant to the deconvolution and empirical Bayes problems, but not to the estimation of the distribution of $Y$. The latter is, of course, identified and it can be estimated by the GMLE which is weakly consistent, whether or not $a=0$ or $a>0$.  It was not estimated consistently by the GMLE when $\scs=\R_-$ and $a=0$, see Theorem \ref{Th:consis}. The problem seems caused by the fact that the GMLE has non-zero density with respect to the counting measure on the observations. The GMLE will be consistent if we avoid this component. This will happen with a few variations. 
  
  For one, we may restrict attention to $0<a$. In that case, the density is bounded by $1/a$, the GMLE is simply MLE, and then argument similar to the consistency of the MLE for a fixed noise scale (cf. \cite{ritov2024} for references) would apply. 
  
  Another possible variation is keeping $\sci=\R_-$ and $\scs=[0,b]$, but  considering only the positive observations, by either truncation---sampling from $F(Y\mid Y>0)$, or censoring---observing $Y_i\ind(Y_i>0)$, $i=1,\dots,n$. The log-likelihood functions are
  \eqsplit{
    \ell_C &= N_-\log\sum_j p_j \Phi\bigl(-x_j/s_j\bigr)+\sum_{Y_i>0}\log\sum_j \frac{p_j}{s_j}\varphi\bigl(\frac{Y_i-x_j}{s_j}\bigr)
    \\
    \ell_T &= \sum_i \log\frac{\sum_j \frac{p_j}{s_j}\varphi\bigl(\frac{Y_i-x_j}{s_j}\bigr)}{\sum_j p_j \bigl(1-\Phi\bigl(-x_j/s_j\bigr)\bigr)}  ,
    }
 for the censored and truncated cases, respectively. It can be argued that the $x_j$s are bounded (for a given sample), and the $s_j$ are bounded away from 0. Therefore it is still a maximization of smooth bounded functions in a family which includes the truth.

\end{remark}

\bibliography{bahadur.bib}

@article{BahadurMLE1958,
 ISSN = {00364452},
 URL = {http://www.jstor.org/stable/25048389},
 author = {R. R. Bahadur},
 journal = {Sankhyā: The Indian Journal of Statistics (1933-1960)},
 number = {3/4},
 pages = {207--210},
 publisher = {Springer},
 title = {Examples of Inconsistency of Maximum Likelihood Estimates},
 urldate = {2024-08-08},
 volume = {20},
 year = {1958}
}

@article{GreenshteinRitov22,
  title={Generalized maximum likelihood estimation of the mean of parameters of mixtures. With applications to sampling and to observational studies},
  author={Greenshtein, Eitan and Ritov, Ya’acov},
  journal={Electronic Journal of Statistics},
  volume={16},
  number={2},
  pages={5934--5954},
  year={2022},
  publisher={The Institute of Mathematical Statistics and the Bernoulli Society}
}

@ARTICLE{JiangZhang09,
  author={Jiang, Wenhua and Zhang, Cun-Hui},
  title =        {General maximum likelihood empirical Bayes estimation of normal means},
  journal =      {Ann. Statist.},
  year =         {2009},
  volume =       {37},
  pages =        {1647--1684},
  }

@ARTICLE{SahaGunt20,
  author =       {Saha, Sujayam  and  Guntuboyina, Adityanand},
  title =        {On the nonparametric maximum likelihood estimator for
gaussian location mixture densities with application to
Gaussian denoising},
  journal =      {Annals of Statistics},
  year =         {2020},
  volume =       {48},
  pages =        {738--762},
}

@misc{ritov2024,
      title={No need for an oracle: the nonparametric maximum likelihood decision in the compound decision problem is minimax}, 
      author={Ya'acov Ritov},
      year={2024},
      eprint={2309.11401},
      archivePrefix={arXiv},
      primaryClass={math.ST},
      url={https://arxiv.org/abs/2309.11401}, 
}

@article{kieferWolfowitz1956,
 ISSN = {00034851},
 URL = {http://www.jstor.org/stable/2237188},
 author = {J. Kiefer and J. Wolfowitz},
 journal = {The Annals of Mathematical Statistics},
 number = {4},
 pages = {887--906},
 publisher = {Institute of Mathematical Statistics},
 title = {Consistency of the Maximum Likelihood Estimator in the Presence of Infinitely Many Incidental Parameters},
 urldate = {2024-04-21},
 volume = {27},
 year = {1956}
}
\end{document}